\documentclass[12pt]{article}
\usepackage{graphicx,amssymb,amsmath,amsthm}
\usepackage{comment,cite,color}
\usepackage{algorithmic}
\usepackage{mathrsfs}
\usepackage[ruled]{algorithm}
\usepackage{time} 
\input epsf
\newtheorem{theorem}{Theorem}

\newtheorem{definition}{Definition}

\newtheorem{remark}{Remark}
\newtheorem{corollary}{Corollary}

\def\Z{\mathbb{Z}}
\def\N{\mathbb{N}}

\DeclareMathOperator{\Span}{span}

\def\Del{\Delta}
\def\tDel{\tilde{\Del}}

\def\be{\begin{equation}}
\def\ee{\end{equation}}

\title{Adaptive non-uniform B-spline dictionaries on a compact interval}
\author{Laura Rebollo-Neira and Zhiqiang Xu\thanks{
Permanent address:
LSEC,  Academy of Mathematics and System Sciences,
Chinese Academy of Sciences, Beijing, 100080, China.}\\
Mathematics, Aston University, Birmingham B4 7ET, UK}
\date{}
\begin{document}
\maketitle
\begin{abstract}
Non-uniform B-spline dictionaries on a compact
interval are discussed. For each given partition,
dictionaries of B-spline functions
for the corresponding spline space  are constructed.
It is asserted that, by dividing the given partition into subpartitions and
joining together the bases for the concomitant
subspaces, slightly redundant dictionaries of B-splines  functions
are obtained. Such dictionaries are proved to
span the spline space
associated to the given partition. The
proposed  construction is  shown to
be potentially useful for the purpose of sparse signal representation.
With that goal in mind, spline spaces specially adapted to produce
a sparse representation of a given signal are considered.
\end{abstract}

\section{Introduction}
A representation in the form of a linear superposition of elements
of a vector space is said to be {\em sparse} if the number of
elements in the superposition is small, in comparison to the
dimension of the corresponding space. The interest for sparse
representations has enormously increased the last few years, in
large part due to their convenience for signal processing techniques
and the results produced by the theory of Compressed Sensing with
regard to the reconstruction of sparse signals from non-adaptive
measurements \cite{cs1,cs2,cs3,cs4,cs5}. Furthermore, the classical
problem of expressing a signal as a linear superposition of elements
taken from an orthogonal basis has been extended to consider the
problem of expressing a signal as a linear superposition of
elements, called  {\em{atoms}}, taken from a redundant set, called
{\em {dictionary}} \cite{Mal98}. The corresponding signal
approximation in terms of highly correlated atoms is said to be {\em
{highly nonlinear}} and has been proved relevant to signal
processing applications. Moreover, a formal mathematical setting for
highly nonlinear approximations is being developed. As a small
sample of relevant literature let us mention
\cite{Dev98,Tem99,GN04}.

In regard to sparse approximations
there are two main problems to be looked at; one is in relation to
the design of suitable algorithms for finding the sparse approximation, 
and other the construction of the dictionaries endowing
the approximation with the property of sparsity. In this
communication we consider the sparse representation matter for the
large class of signals which are amenable to  satisfactory
approximation in spline spaces \cite{Uns99,chuib}. Given a signal, we
have the double goal of a) finding a spline space for approximating
the signal and b) constructing those dictionaries for the space
which are capable of providing a sparse representation of such a
signal. In order to
achieve both aims we first discuss the construction of
dictionaries of B-spline functions for non-uniform partitions, because
the  usual choice, the B-spline basis for the space, is not expected to
yield sparse representations.

In a previous  publication \cite{Bsplinedic}
a prescription for constructing B-spline
dictionaries on the compact interval is advised by restricting
considerations to uniform partitions (cardinal spline spaces).
Since our aim entails to relax this restriction,  we are forced to look at the
problem from a different perspective. Here we divide
the partition into subpartitions and construct the dictionary
by joining together the bases for the subspaces associated to each
subpartition. The
resulting dictionary is proved to span the spline space for the
 given non-uniform partition. Consequently, the uniform case considered in
 \cite{Bsplinedic} arises as a particular case of this general
 construction.
The capability of the proposed nonuniform dictionaries
to produce sparse representations is illustrated by a number of
examples.

The letter is organized as follows: Section 2 introduces splines
spaces and gives the necessary definitions. The property of
splines spaces which provides us with the foundations for the
construction of the proposed dictionaries is proven in this
section (c.f. Theorem 2). For a fixed partition, the actual
constructions of non-uniform B-spline dictionaries is discussed in
Section 3. Section 4 addresses the problem of finding the
appropriate partition giving rise to the spline space suitable for
approximating a given signal. In the same section a number of
examples are presented, which illustrate an important feature of
dictionaries for the adapted spaces. Namely, they may render a
very significant gain in the sparseness of the representation of
those signals which are well approximated in the corresponding
space. The conclusions are drawn in Section 5.

\section{Background and notations}
We refer to the fundamental books \cite{schum,chui,boor} for a complete
treatment of splines.
Here we simply introduce the adopted notation and
the basic definitions which are needed for presenting our
results.
\begin{definition}
Given a finite closed interval $[c,d]$ we define
a {\em partition} of $[c,d]$ as the finite set of points
\begin{equation}\label{Delta}
\Del:=\{x_i\}_{i=0}^{N+1}, N\in\N,\,\,\text{such that} \,\,
c=x_0<x_1<\cdots<x_{N}<x_{N+1}=d.
\end{equation}
We further define $N$ subintervals $I_i, i=0,\dots,N$ as:
$I_i=[x_i,x_{i+1}), i=0,\dots,N-1$ and $I_N=[x_N,x_{N+1}]$.
\end{definition}

\begin{definition}\label{splinespace}
Let $\Pi_{m}$ be the
space of polynomials  of
degree smaller or equal to $m\in\N_0=\N\cup\{0\}$.
Let $m$ be a positive integer and define
\begin{equation}
S_m(\Del)=\{f\in C^{m-2}[c,d]\ : \ f|_{I_i}\in\Pi_{m-1},
i=0,\dots,N\},
\end{equation}
where  $f|_{I_i}$ indicates the
restriction of the function $f$ on the
interval ${I_i}$.
\end{definition}
The standard result established  by the next theorem
is essential for our purpose.
\begin{theorem}\rm(\cite{schum}, pp.111) \label{th:bas}
Let \begin{equation} \Del:=\{x_i\}_{i=0}^{N+1},\,\,
N\in\N,\,\,\text{such that} \,\, c=x_0<x_1<\cdots<x_{N}<x_{N+1}=d.
\end{equation}
Then
$$
S_m(\Delta)\,\,=\,\, {\rm span}\{1,x,\ldots,x^{m-1},
(x-x_i)^{m-1}_+, i=1,\ldots,N\},
$$
where $(x-x_i)^{m-1}_+=(x-x_i)^{m-1}$ for $x-x_i>0$ and $0$
otherwise.
\end{theorem}
We are now ready to prove the theorem from which our proposal will
naturally arise.

\begin{theorem}
Suppose that $\Delta_1$ and $\Delta_2$ are two partitions of
$[c,d]$. It holds to be true that
$$S_m(\Delta_1) +S_m(\Delta_2) = S_m(\Delta_1 \cup \Delta_2).$$
\end{theorem}
\begin{proof}It stems from Theorem 1 and the basic result of linear algebra
establishing  that for $A_1$ and $A_2$ two sets such that
$S_1= \Span\{A_1\}$ and $S_2= \Span\{A_2\}$, one has $S_1 +S_2=
\Span\{A_1\cup A_2\}$.

Certainly, from Theorem 1 and for
$$A_1:=\{1,x,\ldots,x^{m-1},
(x-x_i)^{m-1}_+, x_i \in \Delta_1\}\quad {\text{and}} \quad
A_2:=\{1,x,\ldots,x^{m-1}, (x-x_i)^{m-1}_+, x_i \in \Delta_2\}$$
we have: $S_m(\Delta_1)=\Span\{A_1\}$,\,
$S_m(\Delta_2)=\Span\{A_2\}$. Hence
\begin{eqnarray*}
S_m(\Delta_1) + S_m(\Delta_2)&= &\Span\{A_1 \cup A_2\}\\
&= &\Span \{1,x,\ldots,x^{m-1}, (x-x_i)^{m-1}_+, x_i \in \Delta_1
\cup \Delta_2\}
\end{eqnarray*}
so that,  using Theorem 1 on the right hand side,  the
proof is concluded.
\end{proof}
The next corollary is a  direct consequence of the above theorem.
\begin{corollary}\label{cr:1}
Suppose that $\Delta_j,j=1,\ldots,n$ are  partitions of $[c,d]$.
Then
$$
S_m(\Delta_1) +\cdots + S_m(\Delta_n)\,\,=\,\,S_m(\cup_{j=1}^n\Delta_j).
$$
\end{corollary}
\section{Building B-spline dictionaries}
Let us start by recalling that an
{\em extended partition}
with single inner knots associated with $S_m(\Del)$ is a
set $\tDel=\{y_i\}_{i=1}^{2m+N}$ such that
$$y_{m+i}=x_i,\,\,i=1,\ldots,N,\,\, x_1<\cdots<x_N$$
and the first and last $m$  points $y_1\leq \cdots \leq y_{m} \leq
c,\quad d \leq y_{m+N+1}\leq \cdots \leq y_{2m+N}$ can be
arbitrarily chosen.

With each fixed extended partition $\tDel$ there is associated a
unique B-spline basis for $S_m(\Del)$, that we denote as
$\{B_{m,j}\}_{j=1}^{m+N}$. The B-spline $B_{m,j}$ can be defined
by the recursive formulae \cite{schum}:
\begin{eqnarray*}
B_{1,j}(x)&=& \begin{cases}
1, & t_j\leq x<t_{j+1},\\
0, & {\rm otherwise,}
\end{cases} \\
B_{m,j}(x) &=&
\frac{x-y_j}{y_{j+m-1}-y_j}B_{m-1,j}(x)+\frac{y_{j+m}-x}{y_{j+m}-y_{j+1}}B_{m-1,j+1}(x).
\end{eqnarray*}
The following theorem paves the way for the construction of
dictionaries for $S_m(\Delta)$. We use the symbol $\#$  to indicate
the cardinality of a set.

\begin{theorem}\label{th:dic}
Let $\Delta_j,\,j=1,\ldots,n$ be  partitions of $[c,d]$ and
$\Delta=\cup_{j=1}^n \Delta_j$.  We denote the B-spline basis for
$S_m(\Delta_j)$ as $\{B_{m,k}^{(j)}:k=1,\ldots,m+\#\Delta_j\}$.
Accordingly, a dictionary, ${\mathcal
D}_m(\Delta: \cup_{j=1}^n\Del_j)$, for $S_m(\Delta)$ can be
constructed as
$$
{\mathcal D}_m(\Delta: \cup_{j=1}^n\Del_j)\,\, :=\,\,
\cup_{j=1}^n \{B_{m,k}^{(j)}:k=1,\ldots, m+\#\Delta_j \},
$$
so as to  satisfy
$$
{\rm span}\{{\mathcal D}_m(\Delta: \cup_{j=1}^n\Del_j)
\}\,\,=\,\,S_m(\Delta).
$$
When $n=1$, ${\mathcal D}_m(\Delta:\Del_1)$ is reduced to the
B-spline basis of $S_m(\Delta)$.
\end{theorem}
\begin{proof} It immediately follows from Corollary \ref{cr:1}. Indeed,
\begin{eqnarray*}
 {\rm span}\{{\mathcal D}_m(\Delta: \cup_{j=1}^n\Del_j)\}&
 =&{\rm span}\{\cup_{j=1}^n
\{B_{m,k}^{(j)}:k=1,\ldots, m+\#\Delta_j \} \}\\
&=&S_m(\Delta_1)+ \cdots+S_m(\Delta_1)= S_m(\cup_{j=1}^n
\Delta_j)=S_m(\Delta).
\end{eqnarray*}
\end{proof}
\begin{remark}
Note that the number of functions in the above defined dictionary is
equal to  $n\cdot m+\sum_{j=1}^n\#\Delta_j$, which is larger than
$\dim S_m(\Delta)=m+\#\Delta$. Hence, excluding the trivial case
$n=1$, the dictionary constitutes a redundant dictionary for
$S_m(\Delta)$.
\end{remark}
According to Theorem \ref{th:dic}, to build a dictionary for
$S_m(\Delta)$ we need to choose $n$-subpartitions $\Delta_j\in
\Delta$  such that $\cup_{j=1}^n\Del_j=\Del$. This gives a  great deal of
freedom for the actual construction of a non-uniform B-spline
dictionary. Fig.~1 shows some examples which are produced by
generating a random partition of $[0,4]$ with 6 interior knots.
From an arbitrary  partition
$$
\Del\,\,:=\,\, \{0=x_0<x_1<\cdots <x_{6}<x_{7}=4\},
$$
we generate two subpartitions as
$$
\Del_1\,\,:=\,\, \{0=x_0<x_1<x_3<x_{5}<x_{7}=4\}, \quad
\Del_2\,\,:=\,\, \{0=x_0<x_2<x_4<x_{6}<x_{7}=4\}$$ and join
together the B-spline basis for $\Del_1$ (light lines in the right
graphs of Fig.~1)  and $\Del_2$ (dark lines in the same graphs)
\begin{figure}[!ht]
\begin{center}
\epsfxsize=7cm\epsfbox{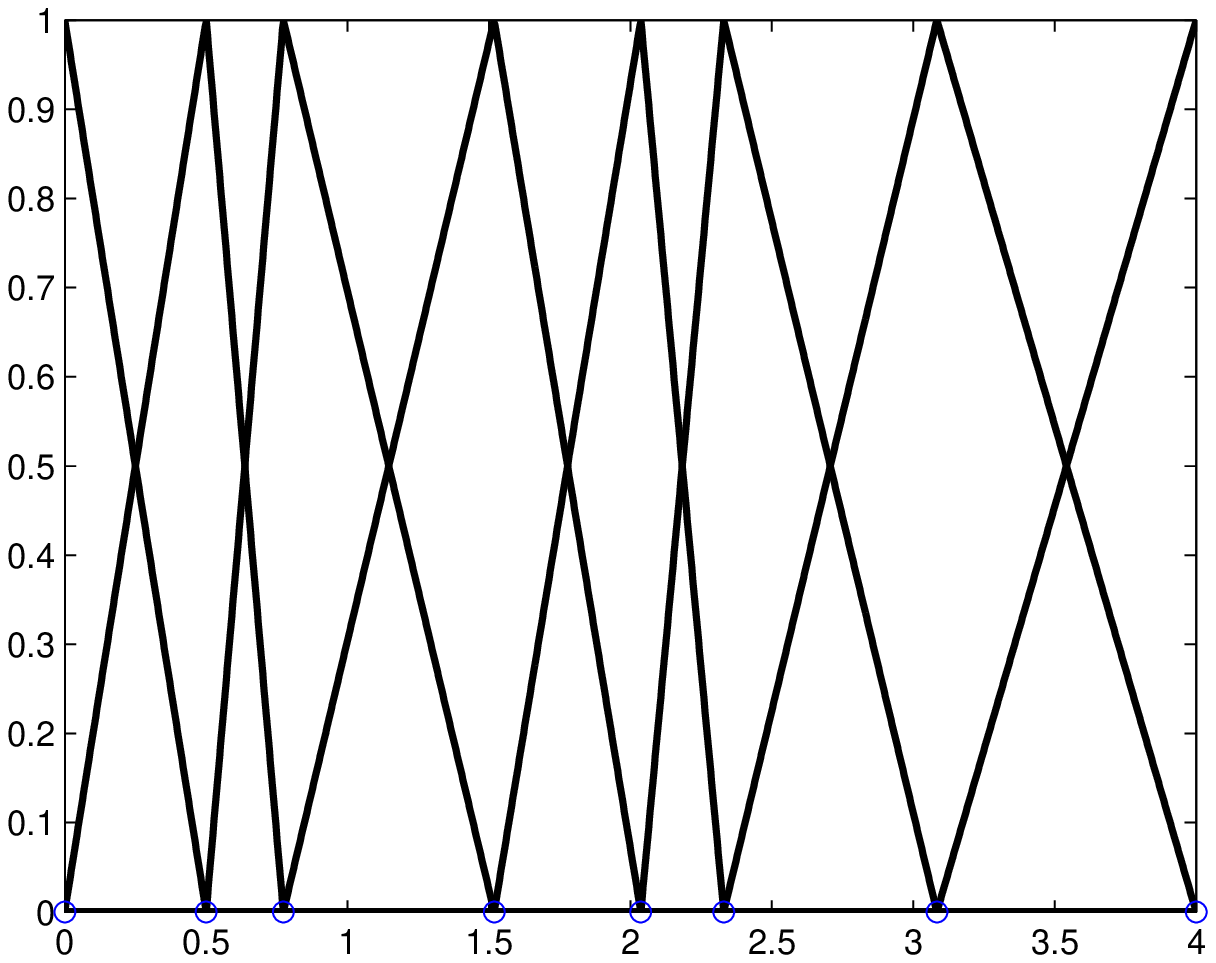}\,\,\,\,\,\,\,\,\, \epsfxsize=7cm\epsfbox{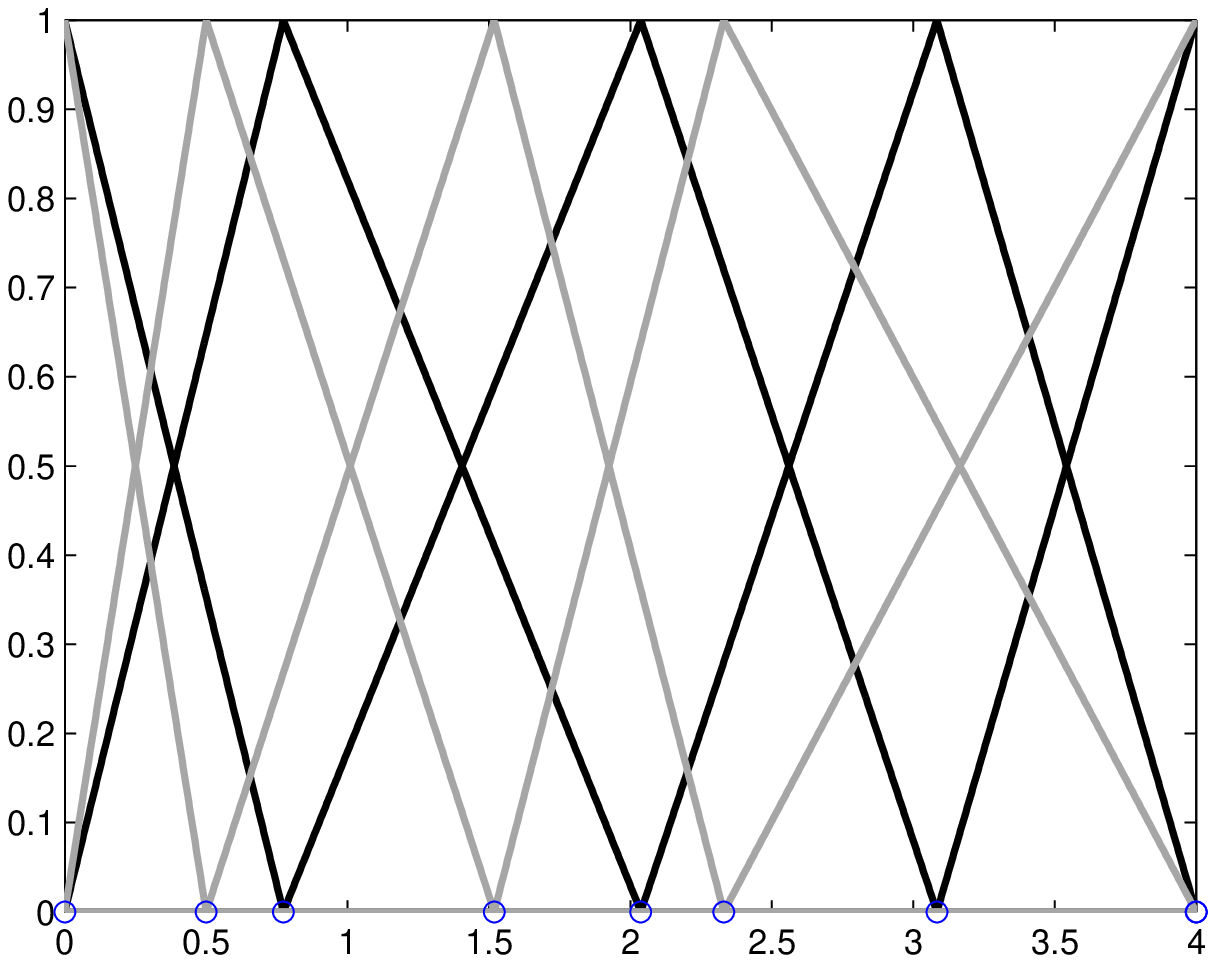}\\
\epsfxsize=7cm\epsfbox{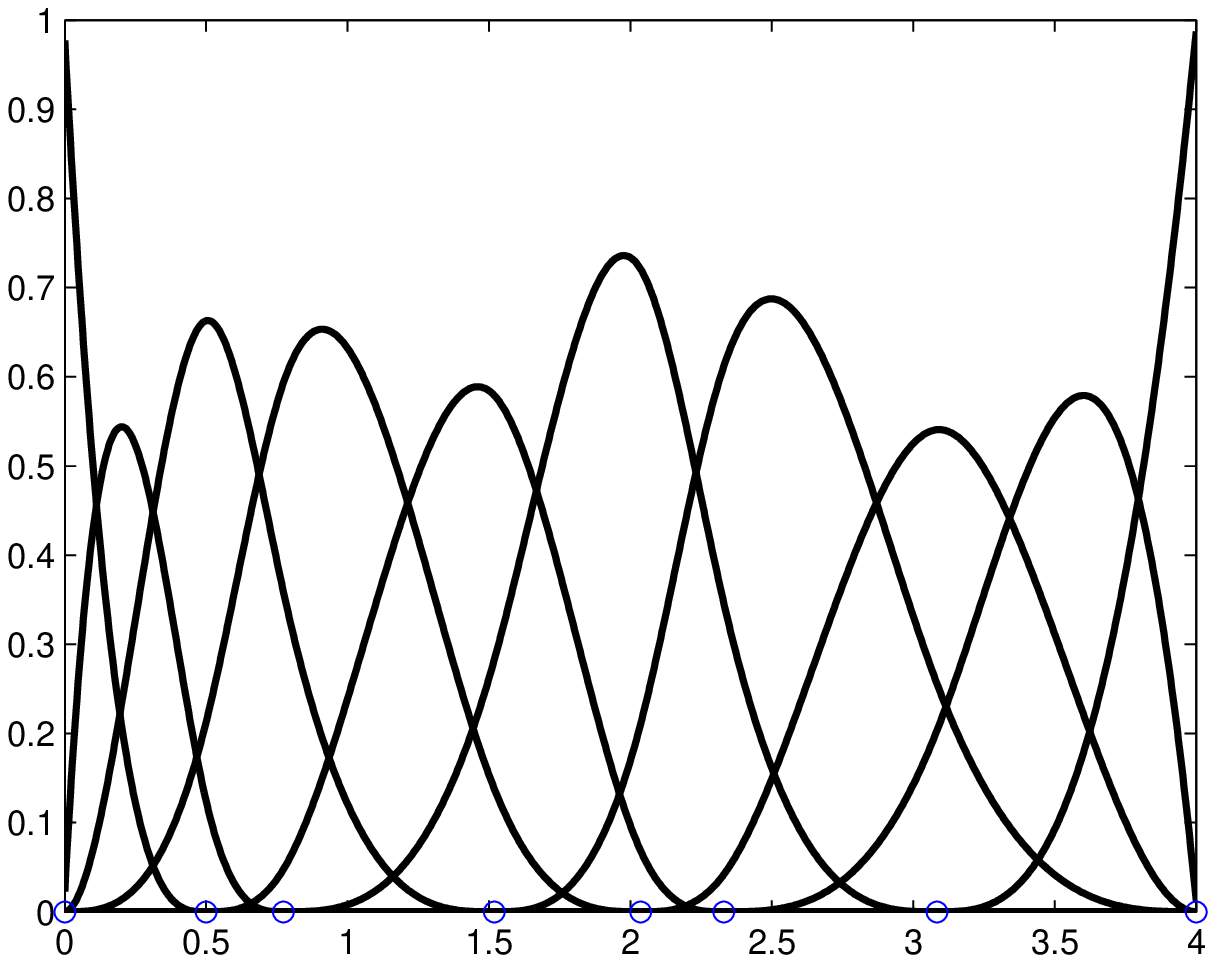}\,\,\,\,\,\,\,\,\,
\epsfxsize=7cm\epsfbox{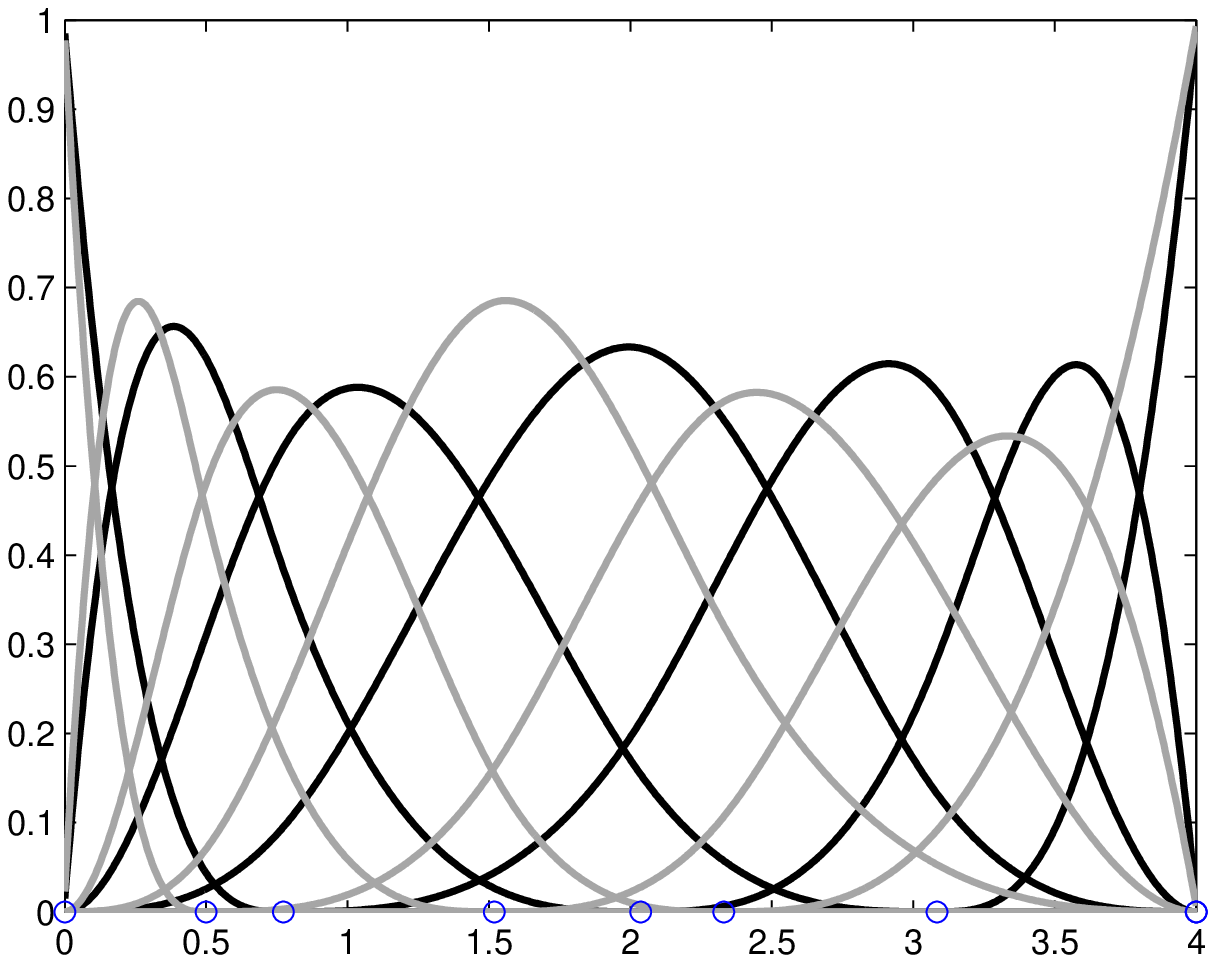}
\end{center}
\caption{\small{Examples of bases (graphs on the left) and the
corresponding dictionaries (right graphs) for a random partition
$\Del$ consisting of 6 interior knots. The top graphs correspond to
linear B-splines ($m=2$). The bottom graphs involve cubic B-splines
($m=4$). The top left graph depicts the B-spline basis for
$S_2(\Del)$ whereas the right one depicts the B-spline dictionary
${\mathcal D}_2(\Del:\Del_1 \cup \Del_2)$ arising by merging the
B-spline basis for $S_2(\Del_1)$ and for $S_2(\Del_2)$ (light and
dark lines, respectively). The bottom graphs have the same
description as the top graphs, but correspond to $S_4(\Del)$ and
${\mathcal D}_4(\Delta:\Delta_1\cup\Delta_2).$}}
\end{figure}
\begin{remark} As should be expected,
the cardinal B-spline dictionaries introduced in \cite{Bsplinedic}
arise here as particular cases of the proposed construction. In
order to show this we use $\Del_b$ to denote an equidistant
partition of $[c,d]$ with distance $b$ between adjacent points such
that $(d-c)/b\in \Z$, i.e., $\Del_b$  is composed by the inner knots
$x_j=c+jb,\,j=1,\ldots, (d-c)/b$. We also consider the partition
$\Del_{j_0,b'}$ with knots $x_{j}=c+j_0b+jb'$, where  $b'/b\in \Z$
and $j_0$ is a fixed integer in $[0,b'/b-1]$. Since
$\cup_{j_0=0}^{{b'/b-1}} \Del_{j_0,b'}=\Del_b$, we can build a
dictionary as
$$
{\mathcal D}_m(\Delta_b: \cup_{j=0}^{b'/b-1} \Del_{j,b'})\,\, =\,\,
\cup_{j_0=0}^{{b'/b-1}}\{B_{m,k}^{(j_0)}:k=1,\ldots,
m+\#\Delta_{j_0,b'} \},
$$
where $B_{m,k}^{(j_0)},k=1,\ldots,m+\#\Del_{j_0,b'}$ is the
cardinal B-spline basis  for the subspace
determined by the partition $\Del_{j_0,b'}$. This
yields a cardinal B-spline dictionary as proposed in
\cite{Bsplinedic}.
\end{remark}
\section{Application to sparse signal representation}
Given a signal, $f$ say,  we address  now the issue of determining a
partition $\Delta$, and sub-partitions $\Delta_j,j=1,\ldots, n$,
such that:  a) $\cup_{j=1}^n\Delta_j=\Delta$ and b) the partitions are
suitable for generating a sparse representation of the
signal in hand. As a first step we propose to tailor the partition
to the signal $f$ by setting $\Del$ taking into account
the critical points of the curvature function
of the signal, i.e.,
$$
T\,\,:=\,\,\{t: \left(\frac{f''}{(1-f'^2)^{3/2}}\right)'(t)=0\}.
$$
Usually the entries in $T$  are choosen as the initial knots of
$\Del$.  In order to obtain more knots we apply subdivision
between consecutive knots in $T$  thereby obtaining a partition
$\Del$ with the decided number of knots.

Because most signals are processed with digital computers, one
normally has to deal with a numerical representation of a signal
in the form of sampling points. Thus, another
problem to be addressed is how to compute the entries in
$T$ from  the sequence of points
$f(kh),k=1,\ldots,n$, where $h$ is the step length in the
discretization. The algorithm below  outlines a procedure for
accomplishing the task.
\begin{algorithm}
\begin{algorithmic}
 \STATE {\bf Input:} the interval $[c,d]$, the discrete signal
$f(c+kh),k=0,\ldots,n$, the subdivision level $l$.
 \STATE {\bf
Output:} The partition $\Del$.
 \STATE $\Del=\{ \}$
  \FOR { $k=2:n-3$} \FOR {$t=0:2$}
  \STATE  $df=(f((k+t+1)h)-f((k+t)h))/h$
  \STATE $ddf=(f((k+t+1)h)+f((k+t-1)h))-2f((k+t)h))/h^2$
   \STATE $c_t=ddf/(1-df^2)^{3/2}$ \ENDFOR

\IF{$|c_0|<|c_1|$  {\rm and} $|c_1|>|c_2|$}
   \STATE $\Del=\Del \cup \{c+(k+1)h\}$
\ENDIF
 \ENDFOR
  \STATE{\rm Set} $N\,\,=\,\,\#\Delta$
\STATE {\rm Set} $\Del\,\,=\,\,\Del\cup \{c,d \}$
 \STATE {\rm Sort the entries in
$\Delta$:}
 $$
c= x_0<x_1<\cdots <x_{N+1}=d
 $$
 \FOR{ $k=0:N$}
\STATE $\Del=\Del\cup \{x_k+t(x_{k+1}-x_k)/l: t=1,\ldots,l-1\}$
\ENDFOR
\end{algorithmic}
\caption{\small{Computing the partition $\Del$: $T(f,l)$.}}
\end{algorithm}
According to Theorem \ref{th:dic}, in order to build a dictionary
for $S_m(\Delta)$ we need to choose $n$-subpartitions
$\Delta_j\in \Delta$  such that $\cup_{j=1}^n\Del_j=\Del$. {{As an
example we suggest a simple method for producing $n$-subpartitions
$\Delta_j\in \Delta$, which is used in the numerical simulations of
the next section. Considering  the partition
$\Del=\{x_0,x_1,\ldots,x_{N+1}\}$ such that $c=x_0<x_1<\cdots
<x_{N+1}=d$, for each integer $j$ in $[1,n]$  we set
$$
\Del_j\,\,:=\,\, \{c,d\}\cup \{x_{k} : k\in [1,N] \mbox{ and }
k\!\!\mod n=j-1\},
$$}}
e.g. if $N=10$ and $n=3$, we have $ \Del_1\,\,=\,\, \{c,  x_3,
x_6, x_9, d\}, \,\,\Del_2\,\,=\,\, \{c, x_1, x_4, x_7, x_{10},
d\}$ $\,\,\,$ $\Del_3\,\,=\,\, \{c, x_2, x_5, x_8, d\}.  $

It is easy to see  that the above defined partitions satisfy
$\cup_{j=1}^n\Del_j=\Del$.

{\subsection{Numerical examples}} We produce here three examples
illustrating the potentiality of the proposed dictionaries for
achieving sparse representations by nonlinear techniques.  The
signal we consider are the following:

\begin{itemize}

\item A chirp signal, $f_1=\cos(2\pi x^2)$, $ x\in [0, 8],$
plotted in the top left of Fig.~1.

\item A seismic signal $f_2$ plotted in the
top right graph of Fig.~1.  This signal was taken
from the WaveLab802 Toolbox.  It is
acknowledged there that the signal is distributed throughout
the seismic industry as a test dataset.

\item A cosine function of random phase $f_3= \cos(8\pi x + \phi(x)),$
where $\phi(x)$ is the piecewise constant function depicted in the
bottom right of Fig.~1. The left graph corresponds to
the signal.

\end{itemize}

\begin{figure}[!ht]
\label{f2}
\begin{center}
\epsfxsize=7.5cm\epsfbox{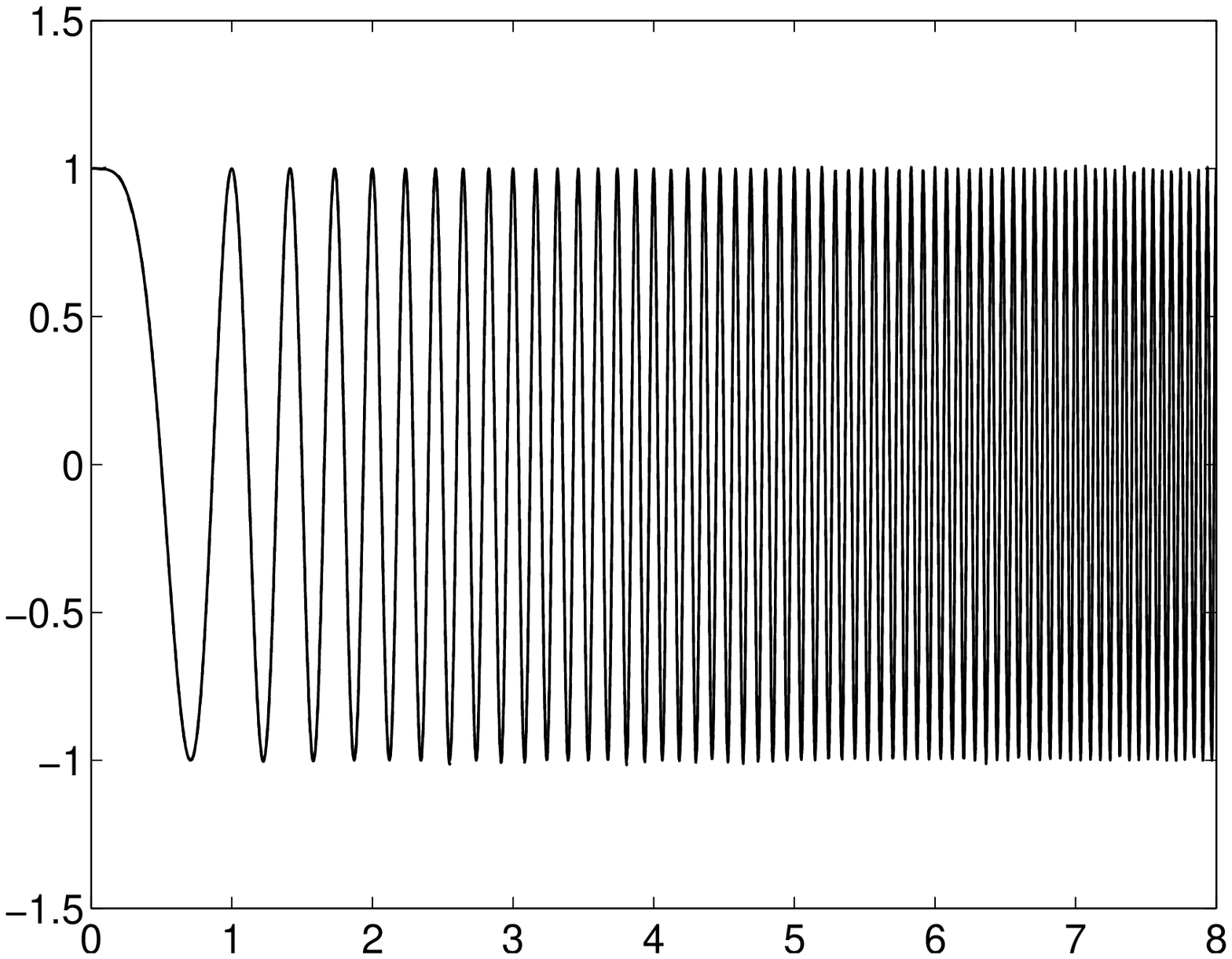}
\epsfxsize=7.5cm\epsfbox{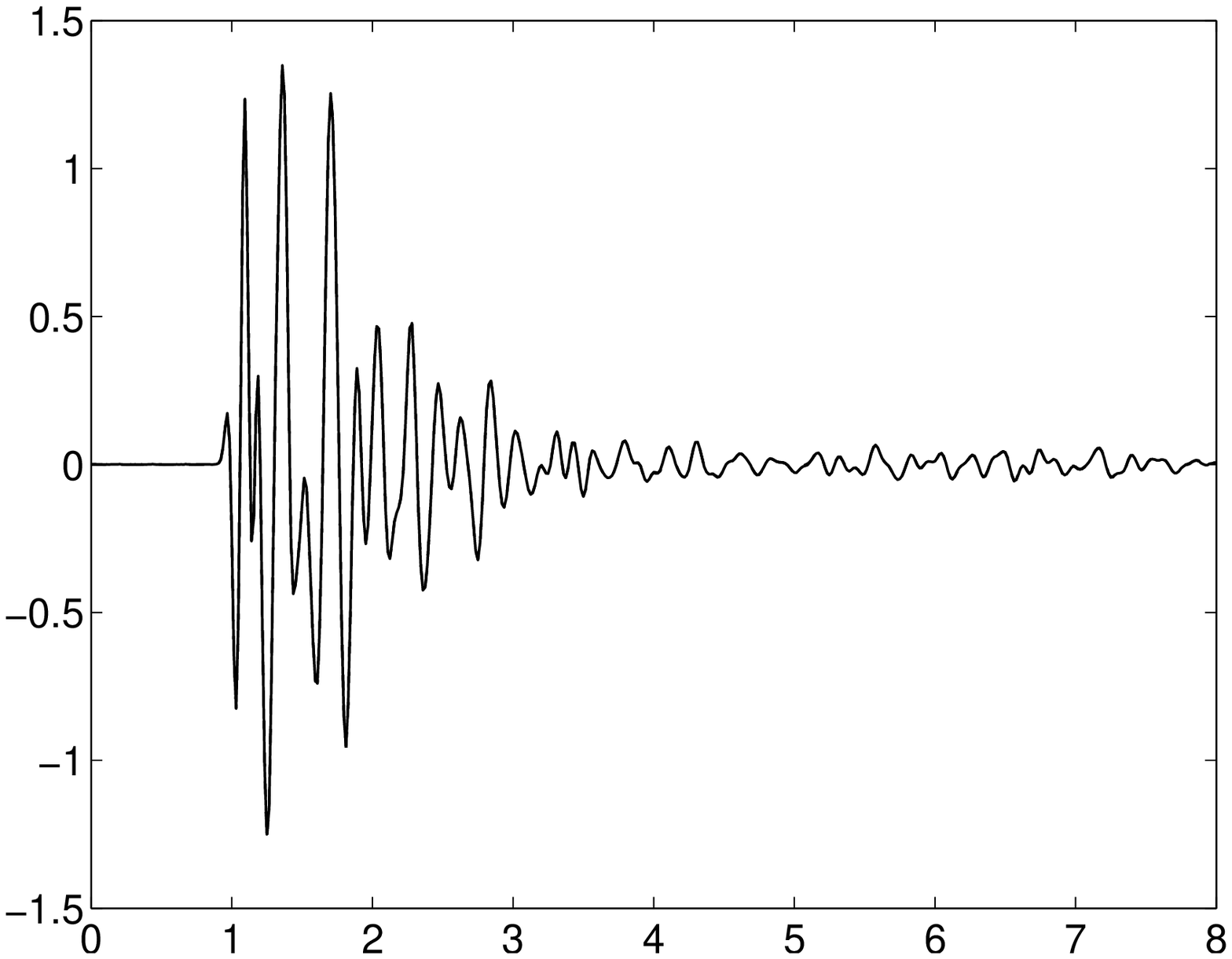}\\
\epsfxsize=7.5cm\epsfbox{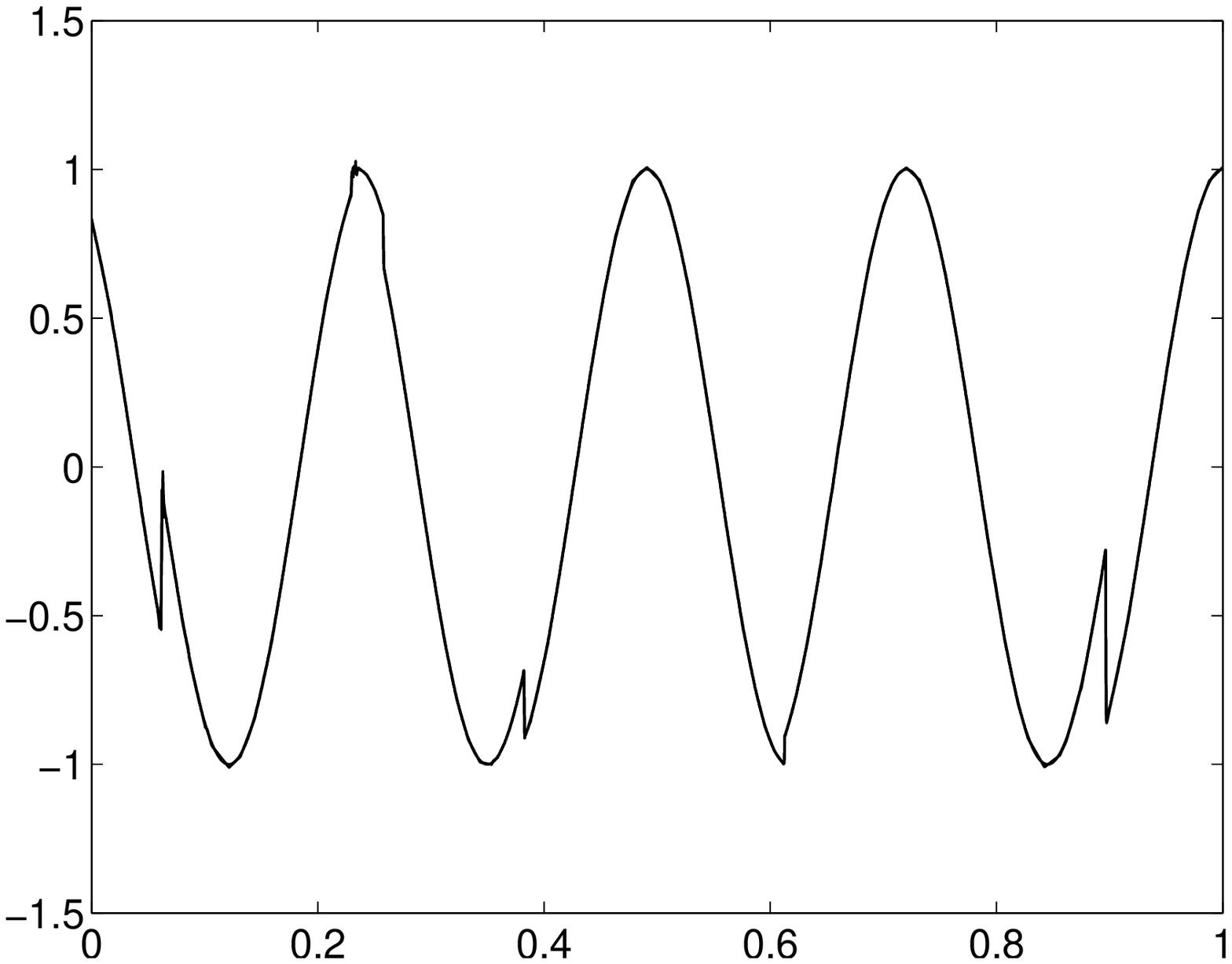}
\epsfxsize=7.5cm\epsfbox{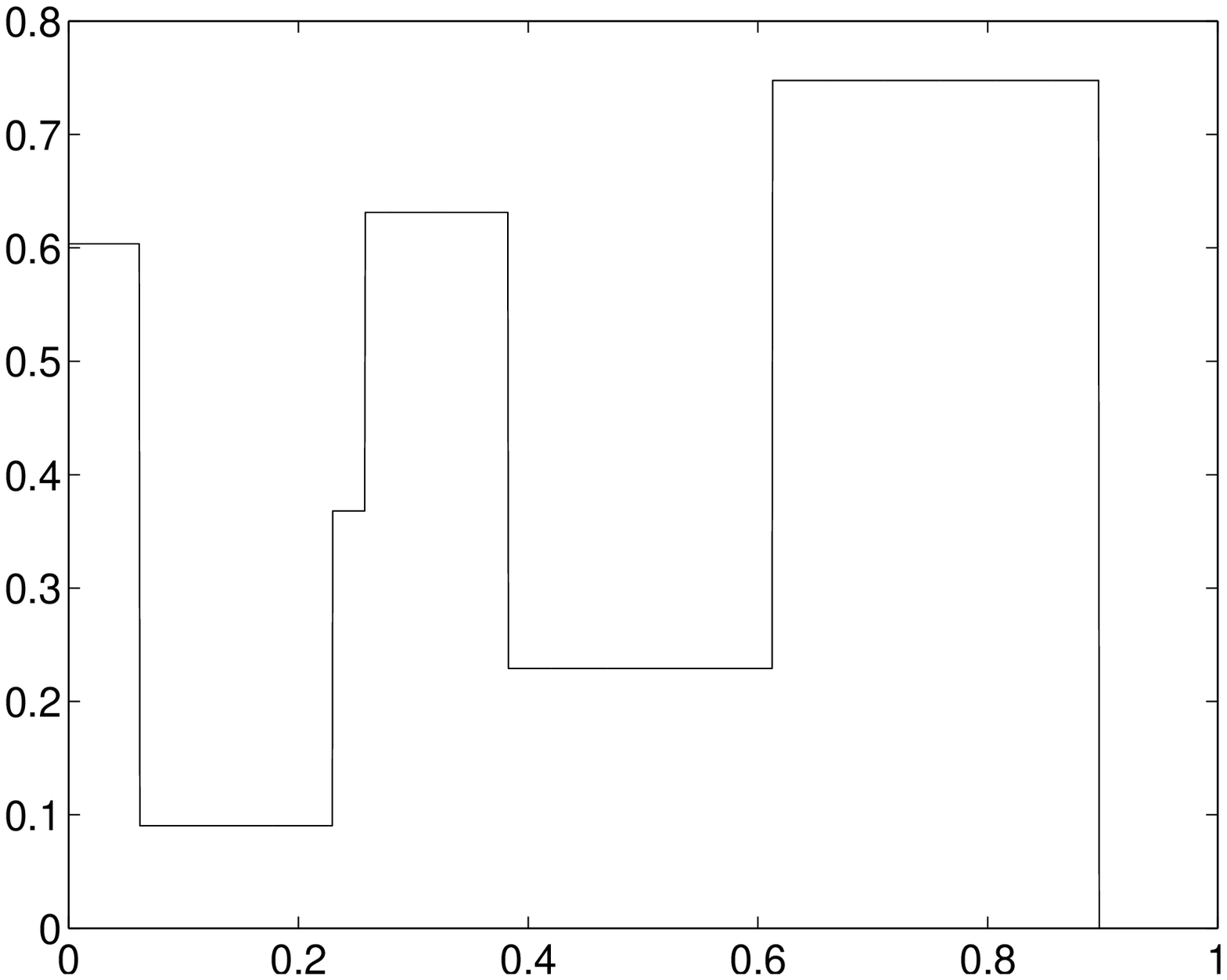}
\end{center}
\caption{\small{Chirp signal $f_1=\cos(2\pi x^2)$ (top left). Seismic
signal $f_2$ (top right). Cosine signal $f_3=\cos(8\pi x +
\phi(x))$ (bottom left)
 for the randomly generate phase $\phi(x)$ plotted in the
bottom right graph. The broken lines, indistinguishable from the
continue lines representing test signals, are the resulting
approximations up to error's norm ${\rm{tol}}_i=0.01||f_i||,\,i=1,2,3$.}}
\end{figure}

The three signals are to be approximated up to a
tolerance ${\rm{tol}}_i=0.01||f_i||,\,i=1,2,3$ for the norm of the
approximation's error.

We deal with the chirp signal $f_1$ on the interval $[0,8]$, by
discretizing it into $L=2049$ samples and applying Algorithm 1 to
produce the partition $\Del= T(f_1,9)$. The resulting number of
knots is $1162$, which is enough to approximate
 the signal, by a cubic B-spline basis for the space, within the above
specified precision ${\rm{tol}_1}$. A dictionary ${\mathcal
D}_4(\Del: \cup_{j=1}^{10}\Del_j)$ for the identical space is
constructed by considering 10 subpartitions, which 
yield  $N_1=1200$ functions.

The signal $f_2$ is a piece of $L=513$ data. A  partition of
cardinality $511$ is obtained as $\Del= T(f_1,8)$ and the dictionary
of cubic splines we have used arises by considering $3$
subpartitions, which yields a dictionary ${\mathcal
D}_4(\Del: \cup_{j=1}^{3}\Del_j)$ of cardinality $N_2= 521$.

The signal $f_3$ is discretized into $L=2049$ samples. The partition
$\Del= T(f_1,43)$ produces $732$ knots. Using $26$ subpartitions we
build a dictionary ${\mathcal D}_2(\Del: \cup_{j=1}^{26}\Del_j)$ of
$N_3=782$ linear B-spline functions.

Denoting by $\alpha^i_n, \,n=1,\ldots,N_i$ the atoms of the
$i$th-dictionary, we look now for the subsets of indices $\Gamma_i,
\,i=1,2,3$ of cardinality $K_i, \,i=1,2,3$ providing us with
a sparse representation of the signals. In other words,
we are interested in the approximations
\begin{eqnarray*}
f_i^{K_i}& =& \sum_{n\in\Gamma_i} c_n^i \alpha^i_n,\quad i=1,2,3,
\end{eqnarray*}
such that $|| f_i^{K_i} - f_i|| \le {\rm{tol}}_i,\ i=1,2,3,$ and the values
$K_i, \,i=1,2,3$ are satisfactory small for the approximation
to be considered sparse. Since the problem of finding
the sparsest solution is intractable, for all the signals we
look for a satisfactory sparse
representation using the same greedy strategy,
which evolves by selecting atoms through stepwise
minimization of the residual error as follows.

i)The atoms are selected one by one according to the Optimized
Orthogonal Matching Pursuit (OOMP) method \cite{oomp} until the
above defined tolerance for the norm of the residual error is
reached.

ii)The previous approximation is improved,  without greatly increasing
 the computational cost, by  a `swapping refinement' which at
each step interchanges one atom of the atomic decomposition with a
dictionary atom, provided that the operation decreases the norm of
the residual error\cite{swapping}.

iii)A Backward-Optimized Orthogonal Matching Pursuit (BOOMP)
method \cite{boomp} is applied to disregard some coefficients of
the atomic decomposition, in order to produce an approximation up
to the error of stage i). The last two steps are repeated until no
further swapping is possible.

Let us stress that, if steps ii) and iii) can be
executed at least once, the above strategy guarantees an improvement
upon the results of OOMP. The gain is
with respect to the number of
atoms involved in the approximation for the given error's norm.

The described technique is applied to all the non-orthogonal
dictionaries we have considered for comparison with the proposed
approach. The results are shown in Table 1. In the first column we
place the dictionaries to be compared. These are: 1) the spline
basis for the space adapted to the corresponding signal, as proposed
in Sec~4.  As already mentioned, for signals $f_1$ and $f_2$ we use
cubic B-splines and for signal $f_3$ the linear one. 2) The
dictionary for the identical spaces consisting of functions of
larger support. 3) The orthogonal cosine bases used by the discrete
cosine transform (dct). 4) The semi-orthogonal cardinal Chui-Wang
spline wavelet basis \cite{CW92} and 5) the Chui-Wang cardinal
spline dictionary for the same space \cite{ARN08}. For signals $f_1$
and $f_2$ we use cubic spline wavelets and for signal $f_3$ linear
spline wavelets.

\begin{table}[!t]
\caption{\small{Comparison of sparsity performance achieved by selecting
atoms from the non-uniform bases  and dictionaries for adapted
spline space (1st and 2nd rows), dft (3rd row), cardinal wavelet
bases and dictionaries (4th and 5th rows).} \label{tab1}}
\centering\vspace*{5mm}
\begin{tabular}{|l|c|c|c|}\hline
Dictionaries & $K^1$ (signal $f_1$)& $K^2$ (signal $f_2$) &$K^3$ (signal $f_3$)\\
\hline \hline
Non-uniform  spline basis  &1097 & 322 & 529 \\
Non-uniform  spline dictionary &173& 129 & 80 \\
\hline
Discreet cosine transform  & 263 & 208 & 669\\
\hline
Cardinal  Chui-Wang wavelet basis  & 246  & 201 &  97 \\
Cardinal  Chui-Wang wavelet  dictionary & 174 & 112 & 92 \\
\hline
\end{tabular}
\end{table}
The last three columns of Table 1 display the number of atoms
involved in the atomic decomposition for each test signal and for
each dictionary. These numbers clearly show a remarkable performance
of the approximation produced by the proposed non-uniform B-spline
dictionaries. Notice that whilst the non-uniform spline space is
adapted to the corresponding signal, only the dictionary for the
space achieves the sparse representation. Moreover
the performance is superior to that of the Chui-Wang
spline wavelet basis \cite{CW92} even for signal $f_3$, which was specially
selected because, due to the abrupt edges delimiting the smooth
lines, is very appropriate to be approximated by wavelets.
It is also worth stressing that  for signal
 $f_1$ and $f_2$ the performance is similar to the
cardinal Chui-Wang dictionary, which is known to render a very good
representation for these signals \cite{ARN08}.  However, whilst
the Chui-Wang cardinal spline wavelet dictionaries introduced in
\cite{ARN08} are significantly redundant with respect to the corresponding
basis (about twice as larger) the non-uniform B-spline dictionaries
introduced here contain a few more functions than the
basis. Nevertheless, as the examples  of this section indicate, the
improvement in the sparseness of the approximation a dictionary
may yield with respect to the B-spline basis is enormous. Still
there is the issue of establishing how to decide on the number of
subpartitions to be considered. For these numerical examples
the number of subpartitions was fixed as the one producing the best
result  when allowing  the number of subpartitions to vary within
some range.  It was observed that only for signal $f_2$ the optimum
number of subpartitions produced results significantly better
that all other values. Conversely, from signals $f_1$ and $f_3$
 some variations from the optimum number of subpartitions  still
produce comparable results.

\section{Conclusions}
Non-uniform B-spline dictionaries for adapted spline spaces have
been introduced. The proposed dictionaries are built by dividing a
given partition into subpartitions and merging the basis for the
concomitant subspaces. The dictionary functions are characterized by
having broader support than the basis functions for the identical
space. The uniform B-spline dictionaries proposed in
\cite{Bsplinedic} readily arise here as a particular case.

The capability of the non-uniform B-spline dictionaries to produce
sparse signal representation has been illustrated by recourse to
numerical simulations. Thus, we feel confident that a number of
applications could benefit from this construction, e.g., we
believe it could also be useful in Computer-Aided Geometry Design
(CAGD), for reducing control points and for finding sparse knot
sets of B-spline curves \cite{cagd,cagd1,cagd2}.
\subsection*{Acknowledgements}
This work has been fully supported by EPSRC, UK, grant
EP$/$D062632$/$1.

\end{document}